\documentclass[11pt]{amsart}

\usepackage{hyperref}
\usepackage{cleveref}
\usepackage[colorinlistoftodos,bordercolor=orange,backgroundcolor=orange!20,linecolor=orange,textsize=scriptsize]{todonotes}

\usepackage{graphicx, enumerate, url}
\usepackage[margin=1in]{geometry}
\usepackage{amssymb,mathtools}
\usepackage{bbm} 
\usepackage{algorithm}
\usepackage{algpseudocode}
\usepackage{color}
\usepackage{tikz}
\usetikzlibrary{3d,calc}

\theoremstyle{plain}
\newtheorem{theorem}{Theorem}

\newtheorem{lemma}{Lemma}

\theoremstyle{definition}

\newtheorem{example}{Example}

\theoremstyle{remark}

\DeclareMathOperator{\tr}{Tr}
\DeclareMathOperator{\diag}{diag}

\newcommand{\N}{\mathbb{N}}
\newcommand{\C}{\mathbb{C}}
\newcommand{\F}{\mathbb{F}}
\newcommand{\Q}{\mathbb{Q}}
\newcommand{\R}{\mathbb{R}}
\newcommand{\Z}{\mathbb{Z}}

\newcommand{\x}{\mathbf{x}}

\newcommand{\y}{\mathbf{y}}

\newcommand{\W}{\mathcal{W}}

\newcommand{\cF}{\mathcal{F}}

\newcommand{\bi}{\mathbbm{i}}

\newcommand{\rB}{\mathrm{B}}

\newcommand{\rH}{\mathrm{H}}

\newcommand{\rL}{\mathcal{L}}

\newcommand{\1}{\mathbf{1}}

\makeatletter
\@namedef{subjclassname@2020}{\textup{2020} Mathematics Subject Classification}
\makeatother

\begin{document}
\title[An SDP characterization of the Crawford number]
{A semidefinite programming characterization of the Crawford number}
\author{Shmuel Friedland}
\address{
 Department of Mathematics, Statistics and Computer Science,
 University of Illinois, Chicago, IL 60607-7045,
 USA, \texttt{friedlan@uic.edu}
 }
 \author{Cynthia Vinzant}
\address{Department of Mathematics,   
University of Washington, Seattle, WA 98195-4350, USA,
  \texttt{vinzant@uw.edu}
 }
 
 \date{March 2, 2026}
  	
\begin{abstract}  
We give a semidefinite programming characterization of the Crawford number. We show that the computation of the Crawford number within $\varepsilon$ precision is computable in polynomial time in the data and $|\log \varepsilon |$.
\end{abstract}
\maketitle

\section{Introduction}\label{sec:intro}

The \emph{numerical range} of an $n\times n$ complex valued matrix  $C$ is defined to be
\begin{equation*}
\begin{aligned}
\W(C)=\{\x^*C\x  \ \text{ s.t. } \  \x\in\C^n, \|\x\|=1\}
\end{aligned}
\end{equation*}
The classical result of Hausdorff-T\"oplitz states that $\W(C)$ is a compact convex subset of $\C\cong \R^2$.  
The \emph{numerical radius} of $C$ is $r(C)=\max\{|\x^*C\x| \text{ s.t. } \x\in\C^n, \|\x\|=1\}$, which gives a vector norm on $\C^{n\times n}$.  

The \emph{Crawford number} of $c\in\C$ is defined as the distance of $c$ to $\W(C)$:
\begin{equation*}
\chi(c,C)=\min\left\{|z-c| \text{ s.t. }  z\in \W(C)\right\}.
\end{equation*}
It is straightforward to check that $\chi(c,C)=\chi(0,C-cI_n)$, where $I_n$ denotes the identity matrix in $\C^{n\times n}$. Thus we restrict to the case $c =0$ and define
\begin{equation}\label{defkapC}
\chi(C)=\chi(0,C) =\min\left\{|z|  \text{ s.t. }  z\in \W(C)\right\}.
\end{equation}

The Crawford number $\chi(C)$ was introduced by Stewart in \cite[Definition 2.1]{Ste79} to study stability of generalized eigenvalue problems.  See \cite{KLV18} for recent results on the Crawford number and its relation to the subspace methods for eigenvalue optimization problems, and \cite{Uhl13}  for a generalized Crawford number.

The aim of this paper to give a semidefinite programming (SDP) characterization of $\chi(C)$:
\begin{theorem}\label{SDPcr}  Let $C = A+\bi B\in\C^{n\times n}$ where $A$, $B$ are Hermitian. Then
\begin{equation}\label{SDPcr1}
\begin{aligned}
\chi(C) \ \ = \ \ \min \ \frac{1}{2}(u+w) \ \  {\rm{ s.t. }}  \ \   \frac{1}{2}(u-w) -\langle A,X\rangle &=0,\\  v-\langle B,X\rangle&=0,\\ \langle I_n, X\rangle &=1,\\
X\succeq 0, \begin{bmatrix}u&v\\v&w\end{bmatrix}&\succeq 0.
\end{aligned}
\end{equation}
Here $X$ is an $n\times n$ Hermitian matrix variable and $u,v,w$ are variable real numbers. 
\end{theorem}

We prove this in Section~\ref{subsec: prft1.1}.
This characterization yields that the computation of the Crawford number within $\varepsilon$ precision is polynomially-time computable in the data and $|\log \varepsilon |$. 
See Section~\ref{subsec:comCr} for details.

\subsection{Notation}
Denote $\Z$, $\Q$, $\Z[\bi]=\Z+\bi \Z$,  and $\Q[\bi]=\Q+\bi\Q$  the sets of integers,  rationals, Gaussian integers and Gaussian rationals respectively.   
Throughout the paper, $\F$ denotes the field of real or complex numbers, $\R$ or $\C$, respectively.
We use $\F^{n\times m}$ to denote the space of  $n\times m$ matrices with entries in $\F$.  
The Frobenius norm of $F\in \F^{n\times n}$ is $\|F\|_F=\sqrt{\tr(F^*F)}$, where $F^*=\bar F^\top$. We identify $\F^{n\times 1}$ with $\F^n$-the space of
column vectors $\x=(x_1,\ldots,x_n)^\top$.   Then $\|\x\|=\sqrt{\x^*\x}$. 
Denote by $\F^{n\times n}\oplus \F^{m\times m}$ the space of the block diagonal matrices $\diag(F,G)\in \F^{(n+m)\times (n+m)}$ where $F\in \F^{n\times n}, G\in \F^{m\times m}$.

Let $\rH_{n}(\F)$ denote the real vector space of self-adjoint $n\times n$ matrices over $\F$. 
This is an inner product space with $\langle F,G\rangle=\tr (FG)$. All of the eigenvalues of a matrix $A\in \rH_{n}(\F)$ are real. We call $A$ \emph{positive semidefinite} or \emph{positive definite}, denoted $A\succeq 0$ and $A\succ 0$, if all of its eigenvalues are nonnegative or positive, respectively. 
Given $A, B\in \rH_{n}(\F)$, $A\succeq B$ and $A\succ B$ mean $A-B\succeq 0$ and $A-B\succ 0$, respectively. We use $ \rH_{n,+}(\F)$ to denote the convex cone of positive semidefinite matrices in $\rH_{n}(\F)$. 
Let $\lambda_n(A)\le \cdots\le \lambda_1(A)$ be the eigenvalues of $A$. 
Denote $[n]=\{1,\ldots,n\}$ for $n\in\N$.

We use the following two decompositions of a matrix $C\in\C^{n\times n}$:
\begin{equation}\label{hdecC}
\begin{aligned}
C&=\Re(C)+\bi \Im(C)  \ \ &&\text{ where } \ \  &\Re(C)& =\frac{C+\overline{C}}{2}, &\Im(C)&=\frac{C-\overline{C}}{2\bi}\in\R^{n\times n},\text{ and }\\
C&=A+\bi B  \ \ &&\text{ where } \ \  &A& =\frac{C+C^*}{2}, &B&=\frac{C-C^*}{2\bi}\in \rH_n(\C).
\end{aligned}
\end{equation}

A \emph{semidefinite program (SDP)} is the problem can that be stated as follows (see \cite{VB96}):  
\[
p^*=\inf\{\langle F_0,X\rangle \text{ s.t. } X\in \rH_{m,+}(\R), \langle F_i, X\rangle = b_i \text{ for } i\in [k]\}
\]
where $F_0,\ldots, F_k\in\rH_m(\R), b_1,\ldots,b_k\in\R$.
This is the problem of minimizing a linear function over the intersection of the positive semidefinite cone with the affine linear subspace 
\begin{equation*}
\begin{aligned}
\rL(F_1,b_1,\ldots,F_k,b_k):=\{X\in \rH_m(\R) \ \text{ s.t. } \  \langle F_i,R\rangle=b_i \text{ for all } i\in[k]\}.\\
\end{aligned}
\end{equation*}
With this notation, the SDP above can be rewritten 
\begin{equation}\label{StSDP}
\begin{aligned}
p^*=\inf\{\langle F_0,X\rangle \text{ s.t. } X\in\rL(F_1,b_1,\ldots,F_k,b_k)\cap \rH_{m,+}(\R)\}.
\end{aligned}
\end{equation}
The set 
$\cF=\rL(F_1,b_1,\ldots,F_k,b_k)\cap \rH_{m,+}(\R)$ is called the \emph{feasible region} of this SDP.

\section{SDP characterization and polynomial computability of $\chi(C)$}\label{sec:prfSDP}
\subsection{Proof of Theorem 1}\label{subsec: prft1.1}
We start with the following simple lemma:
\begin{lemma}\label{SDPdist2} Let $z=x+\bi y\in \C$ where $x,y\in \R$.  Then
\begin{equation}
|z|=\min\left\{r  \ \ \rm{ s.t. } \ \   \begin{bmatrix} r+x&y\\y&r-x\end{bmatrix}\succeq 0\right\}.
\end{equation}
\end{lemma}
\begin{proof}
This $2\times 2$ matrix is positive semidefinite if and only if its diagonal entries $r \pm x$ and determinant $r^2 - x^2-y^2$ are nonnegative. This happens if and only if $r\geq \sqrt{x^2+y^2}$.
The minimum such $r$ is exactly $ \sqrt{x^2+y^2}$. 
\end{proof}

\begin{proof}[Proof of Theorem \ref{SDPcr}.]
We recall that the numerical range has an SDP representation: 
\[
\W(C) = \{\langle C, X\rangle=\langle A,X\rangle+\bi\langle B,X\rangle  \textrm{ s.t. }X\in \rH_{n,+}(\C), \langle I_n, X\rangle =1\}.
\]
See, for example, \cite[Lemma 1]{H10}.
This shows that $z = x+\bi y\in \W(C)$ if and only if there exists $X\in \rH_{n,+}(\C)$ with 
$x = \langle  A, X\rangle$, $y = \langle B, X\rangle $ and $\langle I_n, X\rangle= 1$.

For such a choice of $x,y, X$, let $u = r+x$, $v = y$, and $w = r-x$ where $r\in \R$. 
Note that $r=\frac{1}{2}(u+w)$, $x=\frac{1}{2}(u-w)$.
By Lemma~\ref{SDPdist2}, $\diag\left(X, \begin{bmatrix}u&v\\v&w\end{bmatrix}\right)$ is feasible for \eqref{SDPcr1} if and only if $r \geq |z|$.  
The minimum such $r = \frac{1}{2}(u+w)$ is $|z|$. 
\end{proof}

\subsection{Reformulation as a real symmetric SDP}\label{subsec:SDPchar}
We now show that Theorem \ref{SDPcr} can be stated as a standard SDP problem over $\rH_{2n+2}(\R)$.  We repeat briefly some of the arguments in \cite[Section 2]{FL23}.
For any $C\in \C^{n\times n}$, we define 
\begin{equation}\label{defhatC}
\widehat C=\begin{bmatrix}\Re(C)&-\Im(C)\\\Im(C)&\Re(C)\end{bmatrix}\in\R^{(2n)\times (2n)}.
\end{equation}
It is straightforward to check $C$ is Hermitian if and only $\widehat C$ is real symmetric 
and that, in this case, $C$ is positive semidefinite if and only if  $\widehat C$ is. 
That is, 
$C\in\rH_{n}(\C)$ if and only if $\widehat C\in \rH_{2n}(\R)$ and 
$C\in\rH_{n,+}(\C)$ if and only if $\widehat C\in \rH_{n,+}(\R)$.
Moreover $\langle C, D\rangle =\frac{1}{2}\langle \widehat C, \widehat D\rangle$  for any $C,D\in\rH_n(\C)$. Let 
\begin{equation}\label{defhatHnC}
\begin{aligned}
\widehat \rH_{n}(\C) &=\{\widehat C \text{ s.t. }  C\in\rH_n(\C)\}\subset \rH_{2n}(\R) \  \text{ and }\\
\widehat \rH_{n,+}(\C)&=\{\widehat C \text{ s.t. } C\in\rH_{n,+}(\C)\}\subset \rH_{2n,+}(\R).
\end{aligned}
\end{equation} 
The problem \eqref{SDPcr1} is therefore equivalent to 
\begin{equation}\label{SDPcr2}
\begin{aligned}
\chi(C) \ \ = \ \ \min \ \frac{1}{2}(u+w) \ \  {\rm{ s.t. }}  \ \  (u-w) -\langle \widehat A,Y\rangle =0,  2v-\langle \widehat B,Y\rangle=0, \langle I_{2n}, Y\rangle =2,\\
\diag\left(Y,\begin{bmatrix}u&v\\v&w\end{bmatrix}\right)\in \widehat\rH_{n,+}(\C)\oplus \rH_{2,+}(\R).
\end{aligned}
\end{equation}

The feasible region of this optimization problem is unbounded, as we can replace $u,w$ by $u+s,w+s$ for $s\ge 0$.  As in \cite{FL23}, we replace this the following closely related problem.  We add an additional variable $t$ and constraints that  $t\geq 0$ and $u+w+2t=2(\lceil \|C\|_F\rceil+2)$.

\begin{theorem}\label{thm:SDP_realSym} 
Let $C = A+\bi B\in\C^{n\times n}$ where $A$, $B$ are Hermitian. Then
\begin{equation}\label{SDPcr2a}
\begin{aligned}
\chi(C) \ \ = \ \ \min \ \frac{1}{2}(u+w) \ \  {\rm{ s.t. }}  \ \  (u-w) -\langle \widehat A,Y\rangle &=0,\\ 2v-\langle \widehat B,Y\rangle&=0,\\ \langle I_{2n}, Y\rangle &=2,\\ u+w+2t&=2(\lceil \|C\|_F\rceil+2)\\
\diag\left(Y,\begin{bmatrix}u&v\\v&w\end{bmatrix},t \right)&\in \widehat\rH_{n,+}(\C)\oplus \rH_{2,+}(\R)\oplus \rH_{1,+}(\R).
\end{aligned}
\end{equation}
This is equivalent to a standard SDP of the form 
\begin{equation}\label{admis0}
\begin{aligned}
\min \   \langle F_0, Z\rangle \ \  &\text{ s.t. } \ \ Z\in \cF  \ \ \text{ where } \\
 & \cF=\rL(F_1,0,\ldots,F_{N+2},0,F_{N+3},2, F_{N+4},2(\lceil \|C\|_F\rceil+2))\cap \rH_{2n+3,+}(\R)
\end{aligned}
\end{equation}
for $N = n^2+7n+2$ and some $F_0, \hdots, F_{N+4}\in \rH_{2n+3}(\R)$. Moreover, if $C$ has entries in $\Q[\bi]$, then 
the entries of  $F_j$ can be chosen in $\Q$. 
\end{theorem}
\begin{proof} Let $p^\star$ denote the optimal value of \eqref{SDPcr2a}.  The optimal value of \eqref{SDPcr2}  is $\chi(C)$. 
All constraints of \eqref{SDPcr2} are constraints of  \eqref{SDPcr2a}, showing that $\chi(C) \leq p^\star$.
Conversely,   let
$\diag\left(Y^\star,\begin{bmatrix}u^\star&v^\star\\v^\star&w^\star\end{bmatrix}\right)$ be a feasible solution of  \eqref{SDPcr2} with $\frac{1}{2}(u^\star+w^\star) = \chi(C) +\varepsilon$. 
Such a solution exists for all sufficiently small $\varepsilon \geq 0$. In particular, we can assume $\varepsilon <2$ and let 
\[t^\star=\lceil\|C\|_F\rceil+2-\frac{1}{2}(u^\star+w^\star)=\lceil\|C\|_F\rceil+2-\chi(C)-\varepsilon>0.\]
Here we use that $\chi(C)\le r(C)\le \|C\|_F$.
Then $\diag\left(Y^\star,\begin{bmatrix}u^\star&v^\star\\v^\star&w^\star\end{bmatrix}, t^\star\right)$ is feasible for \eqref{SDPcr2a} with $\frac{1}{2}(u^\star+w^\star) = \chi(C) +\varepsilon$. 
Therefore $p^\star \leq \chi(C)+\varepsilon$. This holds for all sufficiently small  $\varepsilon \geq 0$, showing that $p^\star \leq \chi(C)$ and thus $p^\star = \chi(C)$.

Since $C\mapsto \widehat C$ gives a linear isomorphism between the real vector spaces 
$\rH_{n}(\C)$ and $\widehat \rH_{n}(\C)$, 
 $\widehat\rH_{n}(\C)$ is a subspace of dimension $n^2$ in $\rH_{2n}(\R)$. Therefore $\widehat \rH_{n}(\C)\oplus \rH_{2}(\R)\oplus \rH_{1}(\R)$ is a subspace of dimension $n^2+4$ in $\rH_{2n+3}(\R)$, 
whose dimension is $(n+2)(2n+3)$. 
Let
\begin{equation*}
N = \textrm{codim}\,(\widehat \rH_n(\C)\oplus \rH_2(\R)\oplus \rH_1(\R))=(n+2)(2n+3)-(n^2+4)=n^2+7n+2.
\end{equation*}
It follows that for some  linearly independent $F_1,\ldots, F_{N}\in\rH_{2n+3}(\R)$,
\begin{equation*}
\widehat \rH_n(\C)\oplus \rH_2(\R)\oplus \rH_1(\R)=\rL(F_1,0,\ldots,F_{N},0)\subset \rH_{2n+3}(\R).
\end{equation*}
One can take $F_1,\ldots, F_{N}$ to have entries in $\{-1,0,1\}$.  
Furthermore, for an element $Z = \diag\left(Y,\begin{bmatrix}u&v\\v&w\end{bmatrix},t \right)$ in $\widehat  \rH_n(\C)\oplus \rH_2(\R)\oplus \rH_1(\R)$, we have
\begin{equation*}
\begin{aligned}
\tfrac{1}{2}(u+w) & = \langle F_{0},Z\rangle &&\text{ where } & F_{0}&=\diag\left({\bf 0}_{(2n)\times (2n)}, \tfrac{1}{2}, \tfrac{1}{2}, 0 \right)
\\
 (u-w) -\langle \widehat A,Y\rangle&=\langle F_{N+1},Z\rangle &&\text{ where }  &F_{N+1}&= \diag(-\widehat A,1,-1,0),\\ 
 2v-\langle \widehat B,Y\rangle&=\langle F_{N+2},Z\rangle &&\text{ where } &F_{N+2}&=\diag\left(-\widehat B, \begin{bmatrix}0&1\\1&0\end{bmatrix},0\right),\\
  \langle I_{2n},Y\rangle &=\langle F_{N+3} ,Z\rangle  &&\text{ where } &F_{N+3}&=\diag(I_{2n},0,0,0),\\ 
 u+w+2t &=\langle F_{N+4} ,Z\rangle  &&\text{ where } &F_{N+4}&=\diag\left({\bf 0}_{(2n)\times (2n)}, 1, 1, 2 \right).
\end{aligned}
\end{equation*}
This shows the equivalence of \eqref{SDPcr2a}  and \eqref{admis0}.
If $C$ has entries in $\Q[\bi]$, then $\widehat A$, $\widehat B\in \Q^{(2n)\times(2n)}$, showing that the entries of $F_{N+1}$ and $F_{N+2}$ are in $\Q$. The entries of all other $F_j$ can be chosen to belong to $\{-1,0,\frac{1}{2}, 1, 2\} \subset \Q$.
\end{proof}

\begin{example}[$n=2$] \label{ex:n=2}
Consider the $2\times 2$ matrix $\tilde{C} =  \begin{bmatrix}0& -4 \bi\\ 2 &0 \end{bmatrix}$ and $c = -3 - \bi$. To compute $\chi(c,\tilde{C})$, we compute $\chi(C)$ for 
$C = \tilde{C} - cI_2 =  \begin{bmatrix}3 + \bi & -4 \bi\\  2 & 3 + \bi \end{bmatrix}$. 
This can be written as $A + \bi B$ where $A$, $B$ are Hermitian. Specifically this gives: 
\[{\small
A = \begin{bmatrix}3 &1 - 2 \bi\\ 1 + 2 \bi & 3 \end{bmatrix}\!\!,  \  B = \begin{bmatrix}1 & \!-2 + \bi \\ \!-2 - \bi \! & 1 \end{bmatrix}\!\!,  \
\widehat A = \begin{bmatrix}3 &1 & 0 & 2\\ 1 & 3  & -2 & 0\\ 0 & -2 & 3 &1 \\  2 & 0 & 1 & 3  
 \end{bmatrix}\!\!,  \
\widehat B= \begin{bmatrix}1 &\!-2 \!& 0 & -1\\ \!-2\! & 1  & 1 & 0\\ 0 & 1 & 1 &\!-2\! \\  -1 & 0 & \!-2 \!& 1
 \end{bmatrix}\!.}
\]
The computation of $\chi(C)$ in \eqref{SDPcr2a} is an SDP over $\rH_7(\R)$. 
The subspace $\widehat \rH_2(\C)\oplus \rH_2(\R)\oplus \rH_1(\R)$ of $\rH_7(\R)$
has dimension $8$ and codimension $N = 20$. 
Let $E_{ij}\in \rH_7(\R)$ denote the matrix with $(i,j)$th and $(j,i)$th entries equal to one and all other entries equal to zero. 
We can write $\widehat \rH_2(\C)\oplus \rH_2(\R)\oplus \rH_1(\R)$ 
as $\rL(F_1,0,\ldots,F_{20},0)$ for 
\begin{equation*}
\begin{aligned}\{F_1,\ldots,F_{20}\} = & \{E_{ij}: 1\le i\le 4, 5\le j \le 7 \}\cup \{E_{i7}: 5\le i\le 6\}
\cup \{E_{i(2+i)} : i=1,2 \} \\ & \cup \{E_{14}+E_{23}\} \cup 
\{E_{ij} - E_{(2+i)(2+j)}: 1\le i\le 2, i\le j \le 2 \}.
\end{aligned}
\end{equation*}
We can then take $F_0$, $F_{21}, \hdots, F_{24}\in \rH_7(\R)$ as in the proof of Theorem~\ref{thm:SDP_realSym}, all of which will have entries in $\{0,\frac{1}{2}, \pm1, \pm 2, \pm 3\}$.
The semidefinite program given in \eqref{SDPcr2a} then computes $\chi(C) \approx 1.923$. 
The numerical range of $C$ and unique point $z\in \W(C)$ with $|z|=\chi(C)$ are shown in Figure~\ref{fig:ex}. 
\end{example}

\begin{figure}
\includegraphics[height=1.5in]{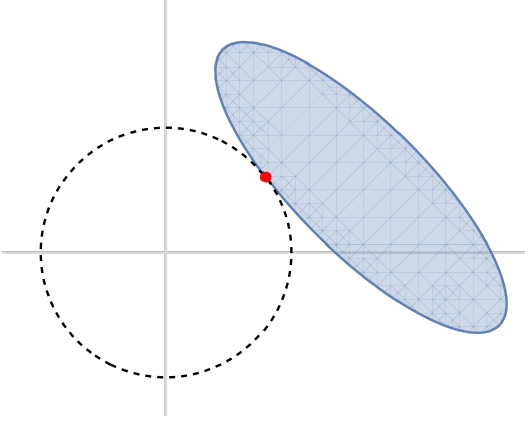} 
\caption{The numerical range from Example~\ref{ex:n=2} and $z\in \W(C)$ minimizing $|z|$.}
\label{fig:ex}
\end{figure}

\subsection{Complexity results for the Crawford number}\label{subsec:comCr}
A polynomial-time bit complexity (Turing complexity) of the SDP problem \eqref{StSDP}
 using the ellipsoid method was shown by  Gr{\"o}tschel-Lov\'asz-Schrijver \cite{GLS81}, under certain conditions, that are stated below.   It used the ellipsoid method of Yudin-Nemirovski \cite{YN76}, and  was inspired by the earlier proof of Khachiyan \cite{Kha79} of the polynomial-time solvability of linear programming.

Let
$\rB(G,r)=\{Y\in\R^{m\times m}: \|Y-G\|_F\le r\}$ denote the ball of radius $r\ge 0$ centered at  $G\in\R^{m\times m}$.
The assumptions of  Gr{\"o}tschel-Lov\'asz-Schrijver \cite{GLS81}  for the SDP \eqref{StSDP} are: 
\begin{equation}\label{pcond}
\begin{aligned}
\text{(1) } &F_0,\ldots,F_k\in \Q^{m\times m}, \ b_1, \ldots, b_k\in \Q\\
\text{(2) }&\text{There exists $G\in \cF\cap \Q^{m\times m}$ with $G\succ 0$ and $0<r\le R$ such that}\\
&\rB(G,r)\cap \rL(F_1,b_1,\ldots,F_k,b_k)\subseteq \cF\subseteq \rB(G,R)\cap \rL(F_1,b_1,\ldots,F_k,b_k).
\end{aligned}
\end{equation}

Then the ellipsoid algorithm finds $X\in \cF\cap \Q^{m\times m}$, such that $\langle F_0,X\rangle\le p^\star +\varepsilon$ \cite[Theorem~1.1]{deklerk_vallentin}.  The complexity of the ellipsoid algorithm is polynomial in the
bit sizes of $G,F_0,\ldots,F_k$, $|\log \varepsilon|$ and $\log \frac{R}{r}$.
Hence, for polynomial-time complexity of $\chi(C)$ we need to assume that 
$C\in\Q[\bi]^{n\times n}$, which occurs if and only if $\widehat C\in\Q^{(2n)\times (2n)}.$

It was shown by  Karmakar \cite{Kar84} that his projective method,  which is related to the Interior Point Method (IPM),  is superior to the ellipsoid method for linear programming problems.   De~Klerk-Vallentin \cite[Theorem 7.1]{deklerk_vallentin} showed polynomial-time complexity for an SDP problem using  IPM method  by applying the short step primal interior point method combined with  Diophantine approximation under the above conditions.

Clearly,  if $C$ is the zero matrix, then $\chi(C)=0$.  Hence we can assume that $n\ge 2$ and $C\ne 0$.   We claim that we can assume that 
$\widehat C\in\Z^{(2n)\times (2n)}\setminus\{0\}$, i.e.,  $C\in \Z[\bi]$.
Indeed, let $\ell\ge 1$ be the products of all denominators (viewed as positive integers) of the rational entries of $\widehat C\ne 0$.  Let $\tilde C=\ell \widehat C$.  Then $\chi(\widehat C)=\frac{1}{\ell}\chi(\tilde C)$.  The following theorem gives an SDP 
characterization for which we can apply the poynomial-time complexity of $\chi(C)$.
\begin{theorem}\label{compCrthm}  
Let $C\in \Q[\bi]^{n\times n}\setminus\{0\}$ and $n\ge 2$.  
Let $(1/n)\tr(C) = x+\bi y$ where $x,y\in \R$.
The SDP \eqref{admis0} computing $\chi(C)$ satisfies  \eqref{pcond} with $r=1/n$,  $R=12+4\lceil\|C\|_F\rceil$ and 
\begin{equation}\label{pcond1}
\begin{aligned}
G=\diag\left(\tfrac{1}{n}I_{2n}, S,1\right) \text{ where } S=\begin{bmatrix}\lceil \|C\|_F\rceil +1+x&y\\ y&\lceil \|C\|_F\rceil +1-x\end{bmatrix}.
\end{aligned}
\end{equation}
In particular, if $C\in \Z[\bi]^{n\times n}$ then the ellipsoid method, or the IPM using the short step primal interior point method combined with  Diophantine approximation, imply that the Turing complexity of $\varepsilon$-approximation of $\chi(C)$ is polynomial in $n$,  the bit sizes of the entries of $C$, and $|\log\varepsilon|$.
\end{theorem}
\begin{proof}  
We write $C = A+\bi B \in \Q[\bi]^{n\times n}$ where $A, B$ are Hermitian. The entries of $A,B$ 
also belong to $\Q[\bi]$ and the entries of $\widehat A,\widehat B$ belong to $\Q$. 
Since the entries of $F_1,\ldots, F_{N}$ were chosen to have entries in $\{-1,0,1\}$, 
all of the matrices $F_0,\ldots, F_{N+4}$ used in \eqref{admis0} belong to $\Q^{(2n+3)\times (2n+3)}$, showing that assumption (1) of \eqref{pcond} is satisfied.

We take $x$, $y$, $r$, $R$, $S$, and $G$ as in the statement of the theorem and  let 
$$\rL = \rL(F_1,0,\ldots,F_{N+2},0,F_{N+3},2, F_{N+4},2(\lceil \|C\|_F\rceil+2))$$
denote the affine linear subspace from \eqref{admis0}. 

First we claim that $G\succ 0$. For this, it suffices to show that $S\succ 0$. 
Note that  $X=\frac{1}{n}I_n$ is positive definite and satisfies $\langle I_n, X\rangle = 1$.  
Therefore $x+\bi y=\frac{1}{n}\tr C = \langle C, X\rangle \in\W(C)$.   
It follows that $|x+\bi y|\le r(C)\le \|C\|_F$. 
In particular, $\|C\|_F\pm x \geq 0$ and $x^2+y^2\leq \|C\|_F^2$, showing that the diagonal entries and determinant of $S-I_2$ are nonnegative.  
It follows that $S -I_2\succeq 0$ and thus $S\succeq I_2 \succ 0$.  

We next show that 
$\rB(G,\frac{1}{n})\cap \rL\subset \cF$.
Let $Z\in \rB(G,\frac{1}{n}) \cap \rL$.  It suffices to show that $Z\succeq 0$. 
Since $\rL \subset \rH_{2n}(\R)\oplus \rH_2(\R)\oplus \rH_1(\R)$, 
\begin{equation}\label{GRrho}
\begin{aligned}
Z=G+W &\text{ where } W=\diag(W_1, W_2, W_3)\in \rH_{2n}(\R)\oplus \rH_2(\R)\oplus \rH_1(\R)\\ 
&\text{ with } \|W\|_F^2 =  \|W_1\|_F^2+\|W_2\|_F^2+\|W_3\|_F^2\le (1/n)^2.
\end{aligned}
\end{equation}
In particular $\|W_j\|_F\le 1/n$ for each $j=1,2,3$.
Note that the block diagonal matrix $Z$ is positive semidefinite if and only if 
each of its blocks, $\frac{1}{n}I_{2n}+W_1$, $S+W_2$ and $1+W_3$, are. 

For any matrix $M\in \rH_m(\R)$, the Frobenius norm $\|M\|_F$ bounds 
the magnitude of all eigenvalues of $M$. 
In particular, if $\|M\|_F \le \mu$, then $\mu I_m +M \succeq 0$. 
It follows that 
\begin{equation*}
\begin{aligned}
\tfrac{1}{n}I_{2n}+W_1\succeq 0, \ \ S+W_2 \succeq I_2+W_2\succeq 0, \  \text{ and }  \ 
1+W_3\succeq 0.
\end{aligned}
\end{equation*}
For the second inequality, we use that $S \succeq I_2$ as shown above.
Thus $Z\succeq 0$.

We now show that $\cF\subseteq \rB(G,R)$.   To do this, we use the trace of a positive semidefinite matrix to bound its Frobenius norm. 
For any matrix $M\in \rH_{m,+}(\R)$ we have $\tr(M)=\sum_{i=1}^m \lambda_i(M)$.
Since $\|M\|_F^2=\sum_{i=1}^m \lambda_i^2(M)$ and $\lambda_i(M)\geq 0$ we find that 
\begin{equation*}
(\tr(M))^2 - \|M\|_F^2 = \sum_{i\neq j} \lambda_i(M) \lambda_j(M) \geq 0.
\end{equation*}
In particular, $ \|M\|_F\le \tr(M)$. 
Now suppose that 
$Z = \diag\left(Y,\begin{bmatrix}u&v\\v&w\end{bmatrix},t \right)\in \cF$. 
Then 
\begin{equation*}\tr(Z)= \tr(Y) + u+w + t \leq 2+ u+w + 2t = 6+2\lceil\|C\|_F\rceil.
\end{equation*}
Here we use the constraints $\tr(Y) = \langle I_n, Y\rangle =2$, $t\geq 0$ and $ u+w + 2t = 2(\lceil\|C\|_F\rceil +2)$. 
By the arguments above, $ \|Z\|_F \leq 6+2\lceil\|C\|_F\rceil$. 
Similarly, $G\in \cF$ and so  $ \|G\|_F \leq 6+2\lceil\|C\|_F\rceil$.
Then by the triangle inequality
\begin{equation*}
\|Z-G\|_F\le \|Z\|_F+\|G\|_F\le 12+4\lceil\|C\|_F\rceil = R.
\end{equation*}
Therefore $\cF\subset \rB(G,R)$.

 Assume that $C\in \Z[\bi]$. Clearly, $\frac{1}{n}\tr C\in \Q[\bi]$.  Hence,  $G\in \Q^{(2n)\times (2n)}.$   The ellipsoid method, or the IPM using the short step primal interior point method combined with  Diophantine approximation, imply that the Turing complexity of $\varepsilon$-approximation of $\chi(C)$ is polynomial in $n$,  the bit sizes of the entries of $C$, and $|\log\varepsilon|$.
\end{proof}

This paper shows that $\chi(C)$ can be computed using the standard software for SDP problems.  This is also the case for the numerical radius and its dual norm, as pointed out in \cite{FL23}.   In numerical simulations in \cite{MO23} it was shown that the known numerical methods for computing $r(A)$ are faster then the SDP methods.  We suspect that the recent numerical methods in \cite{KLV18}  for computing the Crawford  number are superior to the SDP methods applied to the characterizations of the Crawford number discussed in this paper.
\section*{Dedication}
This paper is dedicated to the first author's friend Eitan Tadmor.
\section*{Acknowledgment}  The authors thank the referee for useful remarks.
The work of Shmuel Friedland is partially supported by the Simons Collaboration Grant for Mathematicians. The work of Cynthia Vinzant is partially supported by NSF grant No.~DMS-2153746.

\end{document}